\documentclass[a4paper]{amsart}
\usepackage{amssymb,amsmath,amsthm, amscd, enumerate, mathrsfs}
\usepackage{graphicx}
\usepackage[all]{xy}
\usepackage{tikz}
\usetikzlibrary{intersections, calc}
\usepackage{multicol}

\usepackage{hyperref}

\title[Remarks on bitrational relations of toric Mori fiber spaces]{Remarks on bitrational relations of toric Mori fiber spaces}
\author{Keisuke Miyamoto}
\date{\today, version 0.01}
\address{Department of Mathematics, Graduate School of Science, 
Osaka University, Toyonaka, Osaka 560-0043, Japan}
\email{u901548b@ecs.osaka-u.ac.jp}

\DeclareMathOperator{\Supp}{Supp}
\DeclareMathOperator{\Exc}{Exc}
\DeclareMathOperator{\Pic}{Pic}

\DeclareMathOperator{\PE}{\overline{Eff}}
\DeclareMathOperator{\codim}{codim}

\newtheorem{thm}{Theorem}[section]
\newtheorem{lem}[thm]{Lemma}
\newtheorem{prop}[thm]{Proposition}

\newtheorem{cor}[thm]{Corollary}

\theoremstyle{definition}

\newtheorem{defn}[thm]{Definition}
\newtheorem{rem}[thm]{Remark}
\newtheorem*{ack}{Acknowledgments}

\newtheorem{notation}[thm]{Notation}
\newtheorem{sym}[thm]{Symbols}
\newtheorem*{nac}{Conventions}

\makeatletter
    
    \@addtoreset{equation}{section}
\makeatother

\begin{document}

\begin{abstract}
We can run the MMP for any divisor on any $\mathbb{Q}$-factorial projective toric variety. 
We show that two Mori fiber spaces, which are outputs of the above MMP, are connected by finitely many elementary transforms. 
\end{abstract}

\maketitle 
\tableofcontents


\section{Introduction}\label{m-sec1}

The minimal model program works for $\mathbb{Q}$-factorial projective toric varieties. 
Namely, for any $\mathbb{Q}$-factorial projective toric variety $X$ and any $\mathbb{R}$-divisor $D$ on $X$, 
one can run the $D$-MMP and it ends up with a minimal model or a Mori fiber space (cf. \cite{fujino-sato}). 

The purpose of this paper is to establish the following theorem. 

\begin{thm}\label{main-theorem0}
Let $Z$ be a $\mathbb{Q}$-factorial projective toric variety 
and let $D_Z$ be an $\mathbb{R}$-divisor on $Z$. 
Let $\phi:X\to S$ and $\psi:Y\to T$ be two Mori fiber spaces, 
which are outputs of the $D_Z$-MMP. 

Then the induced birational map $\sigma:X\dashrightarrow Y$ 
is a composition of Sarkisov links (cf.~Definition \ref{links}). 
\end{thm}

Let $X$ be a $\mathbb{Q}$-factorial projective toric variety and 
let $D$ be a Weil divisor on $X$. Then there exists a positive integer $m$ 
such that $mD$ is linearly equivalent to a torus-invariant Cartier divisor on $X$. 
Hence any $\mathbb{R}$-divisor on $X$ is $\mathbb{R}$-linearly equivalent to 
a torus-invariant $\mathbb{R}$-Cartier divisor. 
Therefore, it is sufficient to prove the following statement for Theorem \ref{main-theorem0}. 

\begin{thm}\label{main-theorem}
Let $Z$ be a $\mathbb{Q}$-factorial projective toric variety 
and let $D_Z$ be a torus-invariant $\mathbb{R}$-divisor on $Z$. 
Let $\phi:X\to S$ and $\psi:Y\to T$ be two Mori fiber spaces, 
which are outputs of the $D_Z$-MMP. 

Then the induced birational map $\sigma:X\dashrightarrow Y$ 
is a composition of Sarkisov links. 
\end{thm}

\begin{cor}\label{sub-theorem1}
Let $(Z, \Phi)$ be a $\mathbb{Q}$-factorial projective toric lc pair. 
Let $\phi:X\to S$ and $\psi:Y\to T$ be two Mori fiber spaces, 
which are outputs of the $(K_Z+\Phi)$-MMP. 

Then the induced birational map $\sigma:X\dashrightarrow Y$ 
is a composition of Sarkisov links. 
\end{cor}

We define lc toric pairs in Definition \ref{sing}. 
Corollary \ref{sub-theorem1} immediately follows from Proposition \ref{sing2}. 

\begin{cor}\label{terminal}
Let $Z$ be a $\mathbb{Q}$-factorial projective toric variety with terminal singularities. 
Let $\phi:X\to S$ and $\psi:Y\to T$ be two Mori fiber spaces, 
which are outputs of the $K_Z$-MMP. 

Then the induced birational map $\sigma:X\dashrightarrow Y$ 
is a composition of Sarkisov links. 
\end{cor}

Corollary \ref{terminal} was first established by Matsuki and Sharmov (cf.~\cite[Chapter 14]{matsuki} and ~\cite{acc}). 
This proof is based on the original idea by Sarkisov (cf.~\cite{corti}). 
In their proof, we keep track of three invariants, 
called the Sarkisov degree, associated with the singularities 
and we need to check that the Sarkisov degree satisfies the ascending chain condition. 
This heavily depends on a detailed study of the singularities. 
On the other hand, our approach is quite different from this. 
We use ``the geography of models'' instead of the Sarkisov degree. 
We remark that we can not use the traditional approach by Corti and Matsuki 
as we treat (not necessarily effective) divisors in this paper. 
We note that this idea is based on \cite{hacon-mckernan} and \cite{shokurov}. 

At first, we introduce the notation of Sarkisov links. 

\begin{defn}[Sarkisov links]\label{links}
Let $Z$ be a $\mathbb{Q}$-factorial projective toric variety 
and let $D_Z$ be an $\mathbb{R}$-divisor on $Z$. 
Let $\phi:X\to S$ and $\psi:Y\to T$ be two Mori fiber spaces, 
which are outputs of the $D_Z$-MMP. 

The induced birational map $\sigma:X\dashrightarrow Y$ between $\phi$ and $\psi$ is called 
a {\it Sarkisov link} if it is one of the following four types: 
\begin{multicols}{2}
Type (I)
$$
\xymatrix{
X'\ar@{-->}[r]\ar[d]_p&Y\ar[d]^\psi \\
X\ar[d]_\phi&T\ar[ld] \\
\text{S.}
}
$$
Type (III)
$$
\xymatrix{
X\ar@{-->}[r]\ar[d]_\phi&Y'\ar[d]^q \\
S\ar[rd]&Y\ar[d]^\psi \\
&\text{T.}
}
$$
Type (II)
$$
\xymatrix{
X'\ar@{-->}[r]\ar[d]_p&Y'\ar[d]^q \\
X\ar[d]_\phi&Y\ar[d]^\psi \\
S\ar@{=}[r]&T. 
}
$$
Type (IV)
$$
\xymatrix{ 
X\ar@{-->}[rr]\ar[d]_\phi&&Y\ar[d]^\psi \\
S\ar[rd]_s&&T\ar[ld]^t \\
&\text{R.}
}
$$
\end{multicols}

In the above commutative diagram, the vertical arrows $p$ and $q$ are divisorial contractions 
and the horizontal dotted arrows are compositions of finitely many flops for the $D'_Z$-MMP, 
where $D'_Z$ is an $\mathbb{R}$-divisor on the top left space, that is, $X'$ or $X$. 
The spaces $X'$, $Y'$ and $R$ are realized as the ample models of $\mathbb{R}$-divisors on $Z$ 
(cf.~Definition \ref{am} for definition of ample models). 
In links of Type (I)-(III), $R$ is $\mathbb{Q}$-factorial. 
Moreover, links of Type (IV) are separated into two types: (IV$_m$) and (IV$_s$). 
In a link of Type (IV$_m$), $s$ and $t$ have both Mori fiber structures and $R$ is $\mathbb{Q}$-factorial. 
In a link of Type (IV$_s$), $s$ and $t$ are small birational contractions and $R$ is not $\mathbb{Q}$-factorial. 
We note that a link of Type (IV$_s$) does not appear for $\dim Z\leq3$. 
\end{defn}

\begin{rem}
For toric $3$-folds with terminal singularities, links of Type (I)-(IV$_m$) are completely classified and 
we can find various examples (cf.~\cite{shramov}). 

Next, there is an example of a link of Type (IV$_s$). 
Let $S\to R\leftarrow T$ be a flop for a divisor and we put $X=S\times\mathbb{P}^1$ and $Y=T\times\mathbb{P}^1$. 
Then here is a link of Type (IV$_s$). 
\end{rem}

The contents of this paper are as follows: 
In Section $2$, we quickly recall some basic definitions and properties. 
In Section $3$, we prove Theorem \ref{main-theorem}. 
We remark that all varieties appeared in this section are $\mathbb{Q}$-factorial, projective and toric. 

In this papaer, we will work over an arbitrary algebraically closed field. 

\begin{ack}
The author would like to thank Professor Osamu Fujino 
for various suggestions and warm encouragement. 
The author was partially supported by JSPS 
KAKENHI Grant Number 20J20070. 
\end{ack}


\section{Preliminaries}\label{m-sec2} 

\begin{nac}
A {\it contraction morphism} is a proper morphism $g:X\to Y$ with $g_*\mathcal{O}_X=\mathcal{O}_Y$. 
If $X$ and $Y$ are both normal, 
the condition above is equivalent to the one that $g$ is a surjective morphism with connected fibers. 
A {\it rational contraction} is a rational map $g':X\dasharrow Y'$ if there is a common resolution $p:W\to X$ and $q:W\to Y$ 
which are contraction morphisms. 
A birational map $f:X\dasharrow Z$ is a {\it birational contraction} 
if $f$ is proper and $f^{-1}$ does not contract any divisors. 
We say that $f$ is {\it small} if $f$ and $f^{-1}$ are both birational contractions. 

Let $\mathcal{C}$ be a polytope of a finite-dimensional $\mathbb{R}$-vector space. 
The {\it span} of $\mathcal{C}$ is the span of $\mathcal{C}$ as an affine subspace. 
The {\it relative interior} of $\mathcal{C}$ is the interior of $\mathcal{C}$ 
in the affine space spanned by $\mathcal{C}$. 

We say that an $\mathbb{R}$-vector space $\mathcal{V}_0$  {\it is defined over $\mathbb{Q}$}
if there is a $\mathbb{Q}$-vector space $\mathcal{V}'$ 
such that $\mathcal{V}_0=\mathcal{V}'\otimes_\mathbb{Q}\mathbb{R}$. 
We say that an affine subspace $\mathcal{H}$ of an $\mathbb{R}$-vector space $\mathcal{V}_0$, 
which is defined over  $\mathbb{Q}$,  {\it is defined over  $\mathbb{Q}$} 
if $\mathcal{H}$ is spanned by $\mathbb{Q}$-vectors of $\mathcal{V}_0$. 
\end{nac}

\begin{defn}\label{toric varieties}
Let $N\simeq\mathbb{Z}^n$ be a lattice of rank $n$. 
A {\it toric variety} $X(\Delta)$ is associated to a fan $\Delta$, a finite collection of convex cones 
$\sigma\subset N_{\mathbb{R}}=N\otimes_\mathbb{Z}\mathbb{R}$ satisfying: 
\begin{itemize}
\item[(i)] Each convex cone $\sigma$ is rational polyhedral in the sense that there are finitely many 
$v_1, \ldots, v_s\in N\subset N_\mathbb{R}$ such that
\[
\sigma=\{r_1v_1, \ldots, r_sv_s|r_i\in\mathbb{R}_{\geq0} \text{ for all }i\}
\]
and it is strongly convex in the sense 
\[
\sigma\cap-\sigma=\{0\}. 
\]
\item[(ii)] Each face $\tau$ of a convex cone $\sigma\in\Delta$ is again contained in $\Delta$. 
\item[(iii)] The intersection of two cones in $\Delta$ is a face of each. 
\end{itemize}
\end{defn}

\begin{defn}[Relative Picard numbers]
Let $f:X\to Y$ be a proper morphism between normal varieties. 
We define 
\[
N^{1}(X/Y)=\{\Pic(X)/\equiv_Y\}\otimes\mathbb{R}
\]
and
\[
N_{1}(X/Y)=\{Z_1(X/Y)/\equiv_Y\}\otimes\mathbb{R}, 
\]
where $Z_1(X/Y)$ is the free abelian group of $1$-cycles of $X$ over $Y$. 
These are inducing the following non-degenerate bilinear pairing: 
\[
N^{1}(X/Y)\times N_{1}(X/Y)\to \mathbb{R}. 
\]
It is well-known that $N^{1}(X/Y)$ and $N_{1}(X/Y)$ are finite-dimensional $\mathbb{R}$-vector spaces. 
We write 
\[
\rho(X/Y)=\dim_\mathbb{R}N^{1}(X/Y)=\dim_\mathbb{R}N_{1}(X/Y)
\]
and call it the {\it relative Picard number} of $X$ over $Y$. 
We write $\rho(X)=\rho(X/Y)$ and $N^{1}(X)=N^{1}(X/Y)$ when $Y$ is a point. 
We simply call $\rho(X)$ the {\it Picard number} of $X$. 

If $f$ is a surjective morphism of projective torc varieties with connected fibers, 
then 
\[
\rho(X/Y)=\rho(X)-\rho(Y) 
\]
by \cite[Theorem 6.3.12]{toric-book}. 
For details, see \cite[Lemma 3-2-5 (3)]{kmm}. 
\end{defn}

\begin{defn}[Singularities of pairs]\label{sing}
Let $X$ be a normal variety and $D\geq0$ be an $\mathbb{R}$-divisor on $X$. 
We say that $(X, D)$ is a {\it pair} if $K_X+D$ is $\mathbb{R}$-Cartier. 
In addition, we say that a pair $(X, D)$ is {\it toric} if $X$ is toric and $D$ is consisting of torus-invariant divisors. 

Let $(X, D)$ be a pair. 
Let $f:Y\to X$ be a proper birational morphism from a normal variety $Y$. 
Then we can write 
\[K_Y=f^*(K_X+D)+\sum a_iE_i.\]
We say that $(X, D)$ is {\it klt} (resp.~{\it lc}) 
if $a_i>-1$ (resp.~$a_i\geq-1$) for any $f$ and $i$. 
We say that $X$ is {\it $\mathbb{Q}$-factorial} if every Weil divisor on $X$ is $\mathbb{Q}$-Cartier. 
In addition, we say that a pair $(X, D)$ is {\it $\mathbb{Q}$-factorial} if so is $X$. 

For any subset $S\subset\mathbb{R}$ and any $\mathbb{R}$-divisor $D'$ on $X$, 
we denote $D'\in S$ if all coefficients of the prime decomposition for $D'$ is contained in $S$. 
\end{defn}

\begin{rem}
For any normal toric variety $X$ 
and any torus-invariant $\mathbb{R}$-divisor $D$ on $X$, 
there is a log resolution $f: Y\to X$ of $D$. 
More precisely, $f$ is a projective birational morphism from a smooth toric variety $Y$ such that 
$\Exc(f)\cup f^{-1}(\Supp D)$ is an SNC divisor. 
For details, see \cite[Chapter 11]{toric-book}. 
\end{rem}

The following lemma is the combinatorial characterization of $\mathbb{Q}$-factoriality. 

\begin{lem}{\em (cf.~\cite[Proposition 4.2.7]{toric-book})}
Let $X=X(\Delta)$ be a toric variety. 
Then $X$ is $\mathbb{Q}$-factorial if and only if each of $\sigma\in\Delta$ is simplicial. 
\end{lem}

The following lemma is a well-known criteria for singularities of toric pairs. 

\begin{prop}{\em (cf.~\cite[Proposition 11.4.24]{toric-book})}\label{sing2}
Let $(X=X(\Delta), D)$ be a toric pair. 
If $D\in[0, 1]$, then $(X, D)$ is lc. 
In addition if $D\in[0, 1)$,  then it is klt. 
\end{prop}

\begin{defn}[Ample models]\label{am}
Let $X$ be a normal projective variety and let $D$ be an $\mathbb{R}$-Cartier divisor on $X$. 
Then a rational contraction $g:X\dasharrow Y$ is the {\it ample model} of $D$ if 
\begin{itemize}
\item $Y$ is normal and projective and 
\item there is an ample $\mathbb{R}$-Cartier divisor $H$ such that 
if $p:W\to X$ and $q:W\to Y$ are a common resolution, then 
we can write $p^*D\sim_{\mathbb{R}}q^*H+E$, 
where $E\geq0$, 
then $B\geq E$ for any $B\in|p^*D|_{\mathbb{R}}$. 
\end{itemize}
\end{defn}

For the basic properties of ample models, see \cite[Lemma 3.6.6]{bchm}. 

\begin{rem}
If $f$ is birational, then $E$ is $q$-exceptional. 
If $X$ is toric and $D$ is pseudo-effective, then we can construct the ample model of $D$ in toric geometry. 
By the uniqueness of ample models, $Y$ is always toric. 
We remark that in Section $3$, we always assume that $D$ is pseudo-effective when we treat ample models. 
\end{rem}

\begin{defn}
Let $f:X\dasharrow Y$ be a proper birational contraction of normal varieties and
let $D$ be an $\mathbb{R}$-Cartier divisor on $X$ such that $f_*D$ is also $\mathbb{R}$-Cartier. 
Then we say that $f$ is {\it $D$-non-positive} (resp.~{\it $D$-negative}) 
if there is a common resolution $p:W\to X$ and $q:W\to Y$ such that 
\[
p^*D=q^*f_*D+E, 
\]
where $E\geq0$ is $q$-exceptional (resp.~$E\geq0$ is $q$-exceptional and 
whose support contains the strict transform of the $f$-exceptional divisors). 
\end{defn}

We close this section with definition of minimal models and Mori fiber spaces. 

\begin{defn}\label{models}
Let $f:X\dasharrow Y$ be a birational contraction of normal projective varieties 
and let $D$ be an $\mathbb{R}$-Cartier divisor on $X$ such that $f_*D$ is also $\mathbb{R}$-Cartier. 

We say that $f$ is a {\it weak log canonical model} of $D$ if 
\begin{itemize}
\item $f$ is $D$-non-positive and 
\item $f_*D$ is nef. 
\end{itemize}

We say that $f$ is a {\it minimal model} of $D$ if 
\begin{itemize}
\item $Y$ is $\mathbb{Q}$-factorial, 
\item $f$ is $D$-negative and 
\item $f_*D$ is nef. 
\end{itemize}

Let $\phi:X\to S$ be a contraction morphism of normal projective varieties. 
We say that $\phi$ is a {\it Mori fiber space} of $D$ if 
\begin{itemize} 
\item $X$ is $\mathbb{Q}$-factorial, 
\item $-D$ is $\phi$-ample, 
\item $\rho(X/S)=\rho(X)-\rho(S)=1$ and 
\item $\dim S<\dim X$. 
\end{itemize} 
We say that $\phi$ has a {\it Mori fiber structure} if $\phi$ is a Mori fiber space of some $\mathbb{R}$-Cartier divisor. 

We say that $f$ is the {\it output} of the $D$-MMP 
if $f$ is a minimal model of $D$ 
or a Mori fiber space of $D$. 
On the other hand, we say that $f$ is the {\it result} of running the $D$-MMP 
if $f$ is any sequence of divisorial contractions and flips for the $D$-MMP. 
We emphasize that the result of running the $D$-MMP 
is not necessarily a minimal model of $D$ 
or a Mori fiber space of $D$. 
\end{defn} 


\section{Proof of Theorem \ref{main-theorem}}\label{m-sec3}

In this section, we will closely follow \cite[Section 3, 4]{hacon-mckernan}. 

\begin{sym}
Let $Z$ be a $\mathbb{Q}$-factorial projective toric variety. 
\begin{itemize}
\item $\mathcal{V}(Z)$ is the $\mathbb{R}$-vector space generated by all torus-invariant prime divisors on $Z$.
\end{itemize}
Let $\mathcal{B}$ be a convex polytope in $\mathcal{V}(Z)$. 
Let $f:Z\dasharrow X$ be a birational contraction to a normal projective variety $X$  
and let $g:Z\dasharrow Y$ be a rational contraction to a normal projective variety $Y$. 
\begin{align*}
\mathcal{E}(\mathcal{B})&=\{D_Z\in \mathcal{B}\mid D_Z \text{ is pseudo-effective} \}, \\
\mathcal{M}_f(\mathcal{B})&=\{D_Z\in\mathcal{E}(\mathcal{B})\mid f \text{ is a minimal model of } D_Z\}, \\
\mathcal{A}_g(\mathcal{B})&=\{D_Z\in\mathcal{E}(\mathcal{B})\mid g \text{ is the ample model of } D_Z\}, \\
\mathcal{N}(\mathcal{B})&=\{D_Z\in\mathcal{E}(\mathcal{B})\mid D_Z \text{ is nef}\} 
\end{align*}
and we denote the closure of $\mathcal{A}_g(\mathcal{B})$ by $\mathcal{C}_g(\mathcal{B})$. 
We simply write $\mathcal{A}_g$ to denote $\mathcal{A}_g(\mathcal{B})$ 
if there is no risk of confusion. 
\end{sym}

In this section, we fix the following notation unless otherwise mentioned: 
\begin{itemize}
\item $Z$ is a $\mathbb{Q}$-factorial projective toric variety, 
\item $\mathcal{B}$ is a convex polytope of $\mathcal{V}(Z)$, which is defined over $\mathbb{Q}$. 
\end{itemize}

\begin{prop}\label{partition}
There are only finitely many rational contractions $g_i:Z\dasharrow X_i$ $(1\leq i\leq l)$ such that 
$\mathcal{E}(\mathcal{B})$ is partition by $\mathcal{A}_{g_i}$. 
\end{prop}
\begin{proof}
It follows from the finiteness of minimal models (see~\cite[Theorem 15.5.15]{toric-book}) 
and the property of ample models (cf.~\cite[Lemma 3.6.6]{bchm}). 
\end{proof}

The following two statements come from \cite[Theorem 3.3]{hacon-mckernan} 
and these are easy consequences of the minimal model theory. 
Thus we give simple proofs. 
For the details, see \cite[Theorem 3.3 (2), (3)]{hacon-mckernan}. 

\begin{prop}
With notation as in Proposition \ref{partition}. 
If $\mathcal{A}_{g_j}\cap\mathcal{C}_{g_i}\neq\emptyset$ for $1\leq i, j\leq l$, 
then there is a contraction morphism $g_{i, j}:X_i\to X_j$ such that $g_j=g_{i, j}\circ g_i$. 
\end{prop}
\begin{proof}[Sketch of Proof]
We take $D_Z\in\mathcal{A}_{g_i}$. 
Running the $D_Z$-MMP, we end up with a minimal model $f:Z\dasharrow X$ of $D_Z$. 
Then there is a contraction morphism $g:X\to X_i$ such that $g_i=g\circ f$. 
Using this morphism $g$, we can construct a semi-ample divisor on $X_i$ 
associated to the contraction morphism satisfying the desired property. 
\end{proof}

\begin{prop}\label{span}
With notation as in Proposition \ref{partition}. 
Assume that $\mathcal{B}$ spans $N^{1}(Z)$. 
For any $1\leq i\leq l$, the following are equivalent:
\begin{itemize}
\item[$\bullet$] there is a rational polytope $\mathcal{C}$ contained in $\mathcal{C}_{g_i}$ 
which intersects the interior of $\mathcal{B}$ and spans $\mathcal{B}$.
\item[$\bullet$] $g_i$ is birational and $X_i$ is $\mathbb{Q}$-factorial. 
\end{itemize}
\end{prop}
\begin{proof}[Sketch of Proof]
Suppose that $\mathcal{C}$ span $\mathcal{B}$. 
We take $D_Z$ belonging to the relative interior of 
$\mathcal{C}\cap\mathcal{A}_{g_i}$ belonging to the interior of $\mathcal{B}$. 
Running the $D_Z$-MMP, we end up with a minimal model $f:Z\dasharrow X$ of $D_Z$. 
Then there is certain index $1\leq j\leq l$ such that $f=g_j$. 
Since $D_Z$ belonging to the relative interior of $\mathcal{A}_{g_i}$, we see that $i=j$. 
Thus $g_i=f$ is birational and $X_i=X$ is $\mathbb{Q}$-factorial. 
It is easy to see the converse.  
\end{proof}

The following proposition is the key result of this paper. 

\begin{prop}\label{picard} 
With notation as in Proposition \ref{partition}. 
Assume that $\mathcal{B}$ spans $N^{1}(Z)$. 
If $\mathcal{C}_{g_i}$ spans $\mathcal{B}$ 
and $D_Z$ is a general point of $\mathcal{A}_{g_j}\cap\mathcal{C}_{g_i}$, 
which is also a point of the interior of $\mathcal{B}$ for $1\leq i, j\leq l$, 
then $\rho(X_i/X_j)=\dim\mathcal{C}_{g_i}-\dim\mathcal{C}_{g_j}\cap\mathcal{C}_{g_i}$. 
\end{prop}
\begin{proof} 
Putting $X=X_i$ and $f=g_i$, 
by Proposition \ref{span}, $X$ and $f$ is birational is $\mathbb{Q}$-factorial. 
Let $E_1, \dots, E_k$ be all $f$-exeptional divisors. 
Since $\mathcal{B}$ spans $N^{1}(Z)$, we can take $B_i\in\mathcal{V}(Z)$, 
which are linear combinations of the elements of $\mathcal{B}$, 
such that $B_i\equiv E_i$ and put $B_0=\sum B_i$ and $E_0=\sum E_i$. 
Since $D_Z$ is contained in the interior of $\mathcal{B}$, 
there is a sufficiently small rational number $\delta>0$ 
such that $D_Z+\delta B_0\in\mathcal{B}$. 
Then $f$ is $(D_Z+\delta E_0)$-negative and so it is a minimal model of $D_Z+\delta E_0$ 
and $g_j$ is the ample model of $D_Z+\delta E_0$. 
Thus $D_Z+\delta B_0\in\mathcal{M}_f(\mathcal{B})$ and $D_Z+\delta B_0\in\mathcal{A}_{g_j}$. 
In particular, $D_Z+\delta B_0\in\mathcal{A}_{g_j}\cap\mathcal{C}_f$. 
Since $D_Z$ is general in $\mathcal{A}_{g_j}\cap\mathcal{C}_f$, $D_Z\in\mathcal{M}_f(\mathcal{B})$
and so $f$ is $D_Z$-negative. 

We fix a sufficiently small rational number $\epsilon>0$ such that 
if $D'_Z\in \mathcal{B}$ with $||D'_Z-D_Z||<\epsilon$, 
then $D'_Z\in\mathcal{B}$ and $f$ is $D_Z$-negative. 
Then $D'_Z\in\mathcal{C}_f$ if and only if $D'_{X}=f_*D'_Z$ is nef. 

For any $(a_1, \dots, a_k)\in\mathbb{R}^k$, we put $E=\sum a_iE_i$ and $B=\sum a_iB_i$. 
We put $\mathcal{B}_X=\{D'_X={f}_*D'_Z\mid D'_Z\in \mathcal{B}\}\subset\mathcal{V}(X)$. 
Then $D'_X\in\mathcal{N}(\mathcal{B}_X)$ 
if and only if $D'_{X}+f_*B\in\mathcal{N}(\mathcal{B}_X)$ 
as $D'_Z+B$ is numerically equivalent to $D'_Z+E$. 
This means that 
\[
\mathcal{C}_f\simeq\mathcal{N}(\mathcal{B}_X)\times\mathbb{R}^k
\]
in a neighbourhood of $D_Z$. 

By the above argument and \cite[Lemma 3.6.6]{bchm}, 
$D'_Z\in\mathcal{A}_{g_j}\cap\mathcal{C}_f$ if and only if it holds that $D'_X=f_*D'_Z\in\mathcal{N}(\mathcal{B}_X)$ 
and there is an ample $\mathbb{R}$-Cartier divisor $H$ on $X_j$ such that $f_*D'_Z=(g_{i, j})^*H$, 
where $g_{i, j}:X\to X_j$ is a contraction morphism. 
Since $D_Z\in\mathcal{A}_{g_j}\cap\mathcal{C}_f$, there is an ample $\mathbb{R}$-Cartier divisor $H_0$ on $X_j$ 
such that $f_*D_Z=(g_{i, j})^*H_0$ 
and so there are ample $\mathbb{R}$-Cartier divisors $H_1, \ldots, H_{\rho(X_j)}$, 
whose images on $N^1(X_j)$ are linearly independent, 
such that $f_*D'_Z=(g_{i, j})^*(H_0+\sum b_iH_i)$ 
for any $(b_1, \ldots, b_{\rho(X_j)})\in\mathbb{R}^{\rho(X_j)}$ with $H_0+\sum b_iH_i$ is ample.  
Thus
\[
\dim\mathcal{C}_{g_j}\cap\mathcal{C}_f=k+\rho(X_j). 
\]
Therefore we obtain 
\begin{align*}
\rho(X_i/X_j)&=\rho(X_i)-\rho(X_j)\\
&=\dim\mathcal{N}(\mathcal{B}_X)-\rho(X_j)\\
&=\dim\mathcal{C}_f-\dim\mathcal{C}_{g_j}\cap\mathcal{C}_f. 
\end{align*}
\end{proof}

We recall the following Bertini-type statement for the reader’s convenience. 

\begin{lem}{\em(cf.~\cite[Corollary 3.4]{hacon-mckernan})}\label{bertini}
Let $\mathcal{P}$ be a convex polytope in $\mathcal{V}(Z)$ which spans N$^{\ 1}(Z)$. 
Then for any general affine subspace $\mathcal{H}\subset\mathcal{V}(Z)$, 
the intersection $\mathcal{P}\cap \mathcal{H}$ of $\mathcal{P}$ and $\mathcal{H}$ 
satisfies the conclusions of Proposition \ref{span} and \ref{picard}. 
\end{lem}

\begin{lem}\label{classification}
Assume that $\mathcal{B}$ satisfies the conclusion of Propositions \ref{span} and \ref{picard} 
and that $\dim \mathcal{B}=2$. 
Let $f:Z\dasharrow X$ and $g:Z\dasharrow Y$ be two rational contractions such that 
$\dim\mathcal{C}_f=2$ and $\dim\mathcal{O}=1$, 
where $\mathcal{O}=\mathcal{C}_f\cap\mathcal{C}_g$. 
Assume that $\rho(X)\geq\rho(Y)$ and that $\mathcal{O}$ is not contained in the boundary of $\mathcal{B}$. 
Let $D_Z$ be a point in the relative interior of $\mathcal{O}$ and $D_X=f_*D_Z$. 

Then there is a rational map $\pi:X\dasharrow Y$ such that $g=\pi\circ f$ and either 
\begin{itemize}
\item[(1)] $\rho(X)=\rho(Y)+1$ and $\pi$ is $D_X$-trivial, 
 \begin{itemize}
 \item[(a)] $\pi$ is birational and $\mathcal{O}$ is not contained in the boundary of $\mathcal{E}(\mathcal{B})$, 
  \begin{itemize}
  \item[(i)] $\pi$ is divisorial and $\mathcal{O}\neq\mathcal{C}_g$, 
  \item[(ii)] $\pi$ is small and $\mathcal{O}=\mathcal{C}_g$, 
  \end{itemize}
 \item[(b)] $\pi$ has a Mori fiber structure and $\mathcal{O}=\mathcal{C}_g$ 
 is contained in the boundary of $\mathcal{E}(\mathcal{B})$, 
 \end{itemize}
\item[(2)] $\rho(X)=\rho(Y)$, $\pi$ is a $D_X$-flop 
and $\mathcal{O}\neq\mathcal{C}_g$ is not contained in the boundary of $\mathcal{E}(\mathcal{B})$. 
\end{itemize}
\end{lem}
\begin{proof}
By Proposition \ref{span}, $f$ is birational and $X$ is $\mathbb{Q}$-factorial. 

If $\mathcal{O}$ is contained in the boundary of $\mathcal{E}(\mathcal{B})$, then $\dim\mathcal{C}_g=1$ 
and $\mathcal{O}=\mathcal{C}_g$. 
By Proposition \ref{picard}, there is a contraction $\pi:X\to Y$ which has a Mori fiber structure. 
This is (1, b).

In the rest of proof, we may assume that $\mathcal{O}$ is not contained in the boundary of $\mathcal{E}(\mathcal{B})$. 
If $\dim\mathcal{C}_g=1$, then $\mathcal{O}=\mathcal{C}_g$. 
By Proposition \ref{picard}, there is a contraction $\pi:X\to Y$ with $\rho(X/Y)=1$. 
Since $D_Z$ is not contained in the boundary of $\mathcal{E}(\mathcal{B})$, $D_Z$ is big and so $\pi$ is birational. 
Thus by Proposition \ref{span}, $Y$ is not $\mathbb{Q}$-factorial and so $\pi$ is small. 
This is (1, a, ii). 

We assume that $\dim\mathcal{C}_g=2$. 
Then $g$ is birational and $Y$ is $\mathbb{Q}$-factorial. 
Let $h:Z\to W$ is the ample model of $D_Z$. 
Thus Proposition \ref{picard}, there are two contractions $p:X\to W$ and $q:Y\to W$ 
with $\rho(X/W), \rho(Y/W)\leq1$. 
Then we can explicitly calculate the Picard numbers of X and Y and there are only two cases below: 
\begin{itemize}
\item[(1)] $\rho(X)=\rho(Y)+1$ or
\item[(2)] $\rho(X)=\rho(Y)$. 
\end{itemize}

In (1), $h=g$ and we put $\pi=p$. 
Then $\pi$ is divisorial and this is (1,a, i). 

In (2), $\rho(X/W)=\rho(Y/W)=1$. 
Then $\dim\mathcal{C}_h=1$ since $\dim\mathcal{O}=1$. 
By Theorem \ref{span},  $W$ is not $\mathbb{Q}$-factorial. 
Thus $p$ and $q$ are small and so $\pi$ is $D_X$-flop. 
\end{proof}

\begin{lem}{\em (cf.~\cite[Lemma 3.6]{hacon-mckernan})}\label{result}
Let $f:Z\dasharrow X$ be a birational contraction between $\mathbb{Q}$-factorial projective toric varieties. 
Let $D_Z$ and $D'_Z$ be two torus-invariant $\mathbb{R}$-divisors on $Z$. 
If $f$ is the ample model of $D'_Z$ and $D'_Z-D_Z$ is ample, 
then $f$ is the result of the running the $D_Z$-MMP. 
\end{lem}

Next we will see that certain point contained in the boundary of $\mathcal{E}(V)$ are corresponding to a Sarkisov link. 
Before that we introduce the following additional notation. 

\begin{notation}
Assume that $\mathcal{B}$ satisfies the conclusion of Propositions \ref{span} and \ref{picard} 
and that $\dim \mathcal{B}=2$. 
Let $D^{\dag}_Z$ be a point contained in the boundary of $\mathcal{E}(\mathcal{B})$ 
and the interior of $\mathcal{B}$. 
If $D^{\dag}_Z$ is contained in only one polytope of the form $\mathcal{C}_\bullet$ of two-dimensional, 
then we assume that it is a vertex of $\mathcal{E}(\mathcal{B})$. 

Let $\mathcal{C}_{f_1}, \cdots, \mathcal{C}_{f_k}$ be all two-dimensional rational polytopes 
containing $D^{\dag}_Z$, where $f_i:Z\dasharrow X_i$ are rational maps. 
Note that $f_i$ is birational and $X_i$ is $\mathbb{Q}$-factorial by Proposition \ref{span}. 
Renumbering $\mathcal{C}_{f_i}$ to $\mathcal{C}_i$, 
we may assume that the intersections $\mathcal{O}_0$ and $\mathcal{O}_k$ 
of $\mathcal{C}_1$ and $\mathcal{C}_k$ with the boundary of $\mathcal{E}(\mathcal{B})$, 
$\mathcal{O}_i=\mathcal{C}_i\cap\mathcal{C}_{i+1}$ $(1\leq i\leq k-1)$ 
and all $\mathcal{O}_i$ are one-dimensional. 
We put $f=f_1:Z\dasharrow X=X_1$, $g=f_k:Z\dasharrow Y=X_k$, $X'=X_2$ and $Y'=X_{k-1}$. 
Let $\phi:X\to S=S_0$ and $\psi:Y\to T=S_k$ be the induced morphisms 
and let $h:Z\dasharrow R$ be the ample model of $D^{\dag}_Z$. 

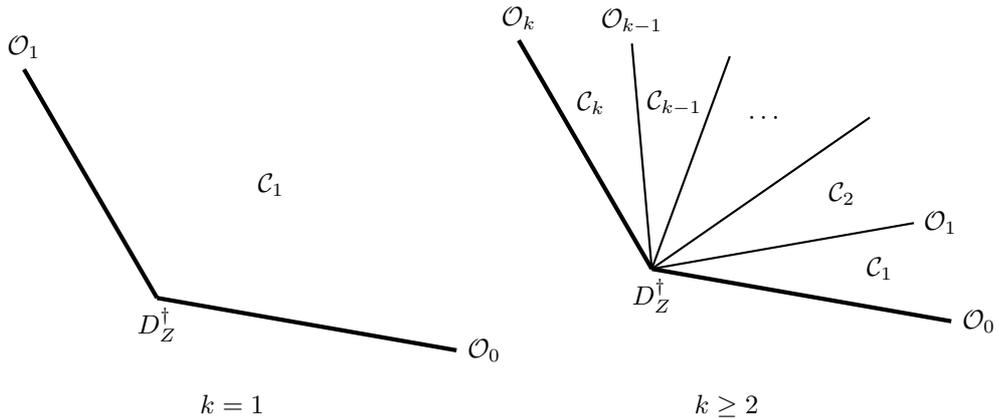
\begin{multicols}{2}
\begin{center}
\begin{tikzpicture}[thick]
\draw (0, 0)node[below]{$D^{\dag}_Z$} -- (120: 3.5)[ultra thick]node[above]{$\mathcal{O}_1$}; 
\draw (0, 0) -- (-10: 4)[ultra thick]node[right]{$\mathcal{O}_0$}; 
\draw (1.5, 1.5) node{$\mathcal{C}_1$}; 
\end{tikzpicture}
\[k=1\]

\begin{tikzpicture}[thick]
\draw (0, 0)node[below]{$D^{\dag}_Z$} -- (120: 3.5)[ultra thick]node[above]{$\mathcal{O}_k$}; 
\draw (0, 0) -- (95: 3)node[above]{$\mathcal{O}_{k-1}$}; 
\draw (0, 0) -- (70: 3); 
\draw (0, 0) -- (35: 3.5); 
\draw (0, 0) -- (10: 3.5)node[right]{$\mathcal{O}_1$}; 
\draw (0, 0) -- (-10: 4)[ultra thick]node[right]{$\mathcal{O}_0$}; 
\draw (3, 0) node{$\mathcal{C}_1$}; 
\draw (2.5, 1) node{$\mathcal{C}_2$}; 
\draw (1.5, 2) node{$\cdots$}; 
\draw (0.3, 2.2) node{$\mathcal{C}_{k-1}$}; 
\draw (-0.8, 2.2) node{$\mathcal{C}_k$}; 
\end{tikzpicture}
\[k\geq2\]
\end{center}
\end{multicols}
\end{notation}

\begin{thm}\label{links}
Let $\mathcal{B}$ and $D^{\dag}_Z$ be notation as above. 
Let $D_Z$ be an $\mathbb{R}$-divisor on $Z$ with $D^{\dag}_Z-D_Z$ is ample. 

Then $\phi$ and $\psi$ are Mori fiber spaces, which are outputs of the $D_Z$-MMP, 
and $f_i$ are the result of running the $D_Z$-MMP. 
Moreover, if $D^{\dag}_Z$ is contained in more than two polytopes, 
then $\phi$ and $\psi$ are connected by a Sarkisov link. 
\end{thm}
\begin{proof} 
By Lemma \ref{classification}, we have the following commutative diagram: 
$$
\xymatrix{
X' =X_2\ar@{-->}[d]_p\ar@{-->}[rr] & & Y'=X_{k-1} \ar@{-->}[d]^q\\
X=X_1 \ar[d]_\phi & & Y=X_k \ar[d]^\psi\\ 
S \ar[rd]_s & & T \ar[ld]^t\\
& R &
}
$$
where $p$, $q$ and the horizontal arrow $X'\dasharrow Y'$ are birational 
and $\phi$ and $\psi$ have Mori fiber structures. 
Since $D^{\dag}_Z-D_Z$ is ample, for any $i$ we can take $D_i\in\mathcal{C}_i$ such that $D_i-D_Z$ is ample. 
By Lemma \ref{result}, $f_i$ is the result of running the $D_Z$-MMP. 
By Proposition \ref{picard}, there is a contraction $X_i\to R$ with $\rho(X_i/R)\leq 2$. 
If $\rho(X_i/R)=0$, then $f_i=h$ and this case does not happen. 
If $\rho(X_i/R)=1$, then $X_i\to R$ gives a Mori fiber structure. 
By Lemma \ref{classification}, $\dim\mathcal{C}_h=1$ 
and there is a facet of $\mathcal{C}_i$ contained in the boundary of $\mathcal{E}(V)$ and so $i=1$ or $k$. 
Therefore if $k\geq3$, then $\rho(X_i/R)=2$ for any $1<i<k$ 
and $X'\dasharrow Y'$ is connected by flops by Lemma \ref{classification} again. 
Moreover since $\rho(X'/R)=2$, $p$ is divisorial and $s$ is the identity, or $p$ is flop and $s$ is not the identity. 
For $q$ and $t$, similar conditions follow and there are only 7 possibilities below: 
\begin{itemize}
\item[(1)] $k=1$, 
\item[(2)] $k=2$, $\rho(X/R)=1$ and $\rho(Y/R)=2$, 
\item[(3)] $k=2$, $\rho(X/R)=2$ and $\rho(Y/R)=1$, 
\item[(4)] $k\geq 3$, $p$ and $q$ are divisorial, and $s$ and $t$ are the identities, 
\item[(5)] $k\geq 3$, $p$ divisorial, $q$ is flop, $s$ is the identity and $t$ is not the identity, 
\item[(6)] $k\geq 3$, $p$ is flop, $q$ is divisorial, $s$ is not the identity and $t$ is the identity, 
\item[(7)] $k\geq 3$, $p$ and $q$ are flops, and $s$ and $t$ are not the identities. 
\end{itemize}

In (1), $X=Y$ and this is a link of Type (IV). 
In (2), $s$ is the identity and $\rho(Y)\geq\rho(X)$. 
Then by Lemma \ref{classification}, there is a divisorial contraction $X'=Y\to X$. 
Thus this is a special case of a link of Type (I). 
In (3), this is similar to (2) and we obtain a special case of a link of Type (III).  
In (4), this is a link of Type (II). 
In (5), this is a link of Type (I). 
In (6), this is a link of Type (III). 
In (7), this is a link of Type (IV). 

The rest of the proof is 
that a link of Type (IV) is splitting into two types (IV$_m$) and (IV$_s$) in (1) and (7). 

We assume that $s$ is a divisorial contraction. 
Then there is a prime divisor $F$ on $S$ which is contracted by $s$. 
Since $\rho(X/S)=1$, there is a prime divisor $E$ on $X$ such that $mE=\phi^*F$ 
for some $m\in\mathbb{Z}_{\geq0}$. 
Since $D^\dag_X=f_*D^\dag_Z$ is numerically trivial over $R$, 
$\mathbb{B}(D^\dag_X+E/R)=E$. 
Since $\rho(X/R)=2$, by the $2$-ray game (cf.~\cite[Chapter 6]{flips}), 
there is a birational contraction  $X\dasharrow V\xrightarrow{f'}W\xrightarrow{g}U$ such that $f'$ is a divisorial contraction 
and $g$ has a Mori fiber space. 
In (1), this is a contradiction since $X=Y$, $\phi$ and $\psi$ are Mori fiber spaces and $\rho(X/R)=2$. 
In (7), $W=Y$ and $U=T$ and so this gives a link of Type (III) and this is a contradiction. 
Similarly, $t$ is not divisorial. 
Thus $s$ and $t$ are not divisorial. 

If $s$ has a Mori fiber structure, then $R$ is $\mathbb{Q}$-factorial and so $t$ has also a Mori fiber structure. 
Thus this is a link of Type (IV$_m$). 

If $s$ is a small contraction, 
then $R$ is not $\mathbb{Q}$-factorial and so $t$ is also small. 
Thus this is a link of Type (IV$_s$). 
\end{proof}

\begin{lem}\label{space}
Let $D_Z$ be a torus-invariant $\mathbb{R}$-divisor on $Z$. 
Let $f:Z\dasharrow X$ and $g:Z\dasharrow Y$ be the results of the $D_Z$-MMP. 
Let $\phi:X\to S$ and $\psi:Y\to T$ be two Mori fiber spaces 
which are outputs of the $D_Z$-MMP. 

Then we can find a two-dimensional convex polytope $\mathcal{B}\subset\mathcal{V}(Z)$, 
which is defined over $\mathbb{Q}$, 
with the following properties: 
\begin{itemize}
\item[(1)] $D'_Z-D_Z$ is ample for any $D'_Z\in\mathcal{E}(\mathcal{B})$, 
\item[(2)] $\mathcal{A}_{\phi\circ f}$ and $\mathcal{A}_{\psi\circ g}$ 
are not contained in the boundary of $\mathcal{B}$, 
\item[(3)] $\mathcal{C}_{f}$ and $\mathcal{C}_{g}$ are two-dimensional, 
\item[(4)] $\mathcal{C}_{\phi\circ f}$ and $\mathcal{C}_{\psi\circ g}$ are one-dimensional, 
\item[(5)] $L=\{D'_Z\in\mathcal{E}(\mathcal{B})\mid D'_Z \text{ is not big}\}$ is connected. 
\end{itemize}
\end{lem}
\begin{proof}
We take ample torus-invariant divisors $H_1, \ldots, H_r\geq0$, which generate $N^1(Z)$, 
and we put $H=H_1+\cdots+H_r$. 
By assumption, there are ample divisors $C$ and $D$ on $S$ and $T$, respectively, such that 
\[
-f_*D_Z+\phi^*C \text{ and } -g_*D_Z+\psi^*D
\]
are both ample. 
Let $\epsilon>0$ be a sufficiently small rational number. 
Then 
\[
-f_*D_Z+\epsilon f_*H+\phi^*C \text{ and } -g_*D_Z+\epsilon g_*H+\psi^*D
\]
are both ample and $f$ and $g$ are both $(D_Z+\epsilon H)$-negative. 
Replacing $H$ by $\epsilon H$, we may assume that $\epsilon=1$. 
We take a torus-invariant $\mathbb{Q}$-divisor $\widetilde{D}_Z$ on $X$ sufficiently close to $D_Z$. 
Then 
\[
-f_*\widetilde{D}_Z+f_*H+\phi^*C \text{ and } -g_*\widetilde{D}_Z+g_*H+\psi^*D
\]
are both ample and $f$ and $g$ are both $(\widetilde{D}_Z+H)$-negative. 
We take torus-invariant $\mathbb{Q}$-divisors $H'_{r+1}$ and $H'_{r+2}$ on $X$ and $Y$, respectively, such that 
\[
H'_{r+1}\in|-f_*\widetilde{D}_Z+f_*H+\phi^*C|_\mathbb{Q}
\text{ and }
H'_{r+2}\in|-g_*\widetilde{D}_Z+g_*H+\phi^*D|_\mathbb{Q}. 
\]
There are torus-invariant $\mathbb{Q}$-divisors $H_{r+1}$ and $H_{r+2}$ on $Z$ such that 
\[
H_{r+1}\sim_\mathbb{Q}f^*H'_{r+1}
\text{ and }
H_{r+2}\sim_\mathbb{Q}g^*H'_{r+2}. 
\]
Let $a>0$ be a sufficiently large rational number and we put a rational convex polytope
\[
\mathcal{B}_0=\left\{\widetilde{D}_Z+a\sum^{r+2}_{i=1} t_iH_i\middle|\sum^{r+2}_{i=1} t_i\leq1, t_i\geq0\right\}. 
\]
Possibly replacing $H_i$ by suitable ones, we may assume that (2) holds for $\mathcal{B}_0$. 

On the other hand, since $f$ is $(\widetilde{D}_Z+H+H_{r+1})$-negative 
and $(\phi\circ f)_*(\widetilde{D}_Z+H+H_{r+1})$ is ample, 
$\widetilde{D}_Z+H+H_{r+1}\in\mathcal{A}_{\phi\circ f}(\mathcal{B}_0)$. 
Similarly, $\widetilde{D}_Z+H+H_{r+2}\in\mathcal{A}_{\psi\circ g}(\mathcal{B}_0)$ 
and $f$ and $g$ are weak log canonical models of $\widetilde{D}_Z+H+H_{r+1}$ 
and $\widetilde{D}_Z+H+H_{r+2}$, respectively.
Thus $\widetilde{D}_Z+H+H_{r+1}\in\mathcal{C}_f(\mathcal{B}_0)$ 
and $\widetilde{D}_Z+H+H_{r+2}\in\mathcal{C}_g(\mathcal{B}_0)$. 

Let $\mathcal{H}_0$ be the translation by $\widetilde{D}_Z$ of the affine subspace generated by $H+H_{r+1}$ and $H+H_{r+2}$ 
and let $\mathcal{H}$ be a small perturbation of $\mathcal{H}_0$, which is defined over $\mathbb{Q}$. 
Putting $\mathcal{B}=\mathcal{B}_0\cap\mathcal{H}$, 
$\mathcal{B}$ satisfies (1) and (2). 
Since $\mathcal{B}_0$ spans $N^1(Z)$, $\mathcal{C}_{\phi\circ f}(\mathcal{B}_0)$ spans $\mathcal{B}_0$ 
and so by Lemma \ref{bertini}, $\mathcal{B}$ satisfies (3). 
By Proposition \ref{picard}, $\dim\mathcal{C}_{\phi\circ f}(\mathcal{B})
=\dim\mathcal{C}_{\psi\circ g}(\mathcal{B})=1$ and so $\mathcal{B}$ satisfies (4). 

Finally we see that we can take $\mathcal{B}$ satisfying (5). 
Since $\phi$ and $\psi$ are Mori fiber spaces, we may assume that $\rho(Z)\geq2$. 
There is a surjective linear map from $\mathcal{V}(Z)$ to $N^1(Z)$. 
Then the pullback of the pseudo-effective cone $\PE(Z)\subset N^1(Z)$ via this map 
is the convex polyhedron $\mathcal{P}$ containing a $(\dim Z)$-dimensional vector subspace $V$ 
since $\dim\mathcal V(Z)=\rho(Z)+\dim Z$. 
Then possibly replacing $H_i$ by suitable ones, 
we can take a two-dimensional rational convex polytope $\mathcal{B}$, 
which does not contan $V$, since $\codim V=\rho(Z)\geq2$. 
Thus $\mathcal{B}$ satisfies (5) as $\mathcal{E}(\mathcal{B})=\mathcal{B}\cap \mathcal{P}$. 
\end{proof}

\begin{proof}[Proof of Theorem \ref{main-theorem}]
We take a two-dimensional rational convex polytope $\mathcal{B}\subset\mathcal{V}(Z)$ given by Lemma \ref{space}. 
We take $D_0\in\mathcal{A}_{\phi\circ f}$ and $D_1\in\mathcal{A}_{\psi\circ g}$ 
belonging to the interior of $\mathcal{B}$. 
As $\mathcal{B}$ is two-dimensional, removing two points $D_0$ and $D_1$, 
the boundary of $\mathcal{E}(\mathcal{B})$ separates into two parts. 
Then one of the two parts of the boundary of $\mathcal{E}(\mathcal{B})$ is 
contained in $L$ by Lemma \ref{space} (5). 
Tracing this part from $D_0$ to $D_1$, 
we obtain finitely many points $D_i$ $(2\leq i\leq k)$ 
which are contained in rational polytopes of two-dimensional. 
By Theorem \ref{links}, each of $D_i$ gives a Sarkisov link 
and $\sigma$ is connected by these links. 
\end{proof}










\end{document}